\documentclass[12pt,reqno]{article}
\usepackage{amssymb,amsmath}
\usepackage{amsthm}
\usepackage[all,cmtip]{xy}
\usepackage{color}
\usepackage[colorlinks=true,
filecolor=webbrown,
citecolor=webwheel]{hyperref}
\definecolor{webwheel}{rgb}{0,.5,0}
\definecolor{webbrown}{rgb}{.6,0,0}
\usepackage{fullpage}
\usepackage{float}
\usepackage{graphics,amsmath,amssymb}
\usepackage{amsthm}
\usepackage{amsfonts}
\usepackage{latexsym}
\usepackage{epsf}
\begin{document}
\theoremstyle{plain}
\newtheorem{theorem}{Theorem}
\newtheorem{corollary}[theorem]{Corollary}
\newtheorem{lemma}[theorem]{Lemma}
\newtheorem{proposition}[theorem]{Proposition}

\theoremstyle{definition}
\newtheorem{definition}[theorem]{Definition}
\newtheorem{example}[theorem]{Example}
\newtheorem{conjecture}[theorem]{Conjecture}

\theoremstyle{remark}
\newtheorem{remark}[theorem]{Remark}
\newcommand{\A}{\mathbb{N}_0}
\newcommand{\ber}{\mathcal{B}}
\newcommand{\eulr}{\mathcal{E}}
\newcommand{\s}{\mathcal{S}}
\title{Subgroups of the additive group of real line}
\date{}
\author{\baselineskip5pt \sf Jitender Singh\\\small \sf Department of Mathematics, \small \sf Guru Nanak Dev
University, \\ \small \sf Amritsar-143005, Punjab, INDIA\\  {\small \tt
sonumaths@gmail.com}} \maketitle
\begin{abstract}
 Without assuming the field structure on  the additive group of  real numbers $\mathbb{R}$ with the usual order $<,$ we explore the fact that, every proper subgroup of $\mathbb{R}$ is  either closed or dense. This property of the subgroups of the additive-group of reals is special and well known (see Abels and Monoussos [4]). However, by revisiting it, we provide  another  direct proof. We also generalize this result to arbitrary topological groups in the sense that, any topological group having this property of the subgroups, in a given topology, is either connected or totally-disconnected.
\end{abstract}
\parskip=.3cm
\baselineskip18pt
\parindent=0cm
By a topological group, we mean, an
abstract group $G$ which is also a topological space where the two maps $\begin{array}{c}
                                                                      G\times G\rightarrow G \\
                                                                      (x,y)\mapsto xy
                                                                    \end{array}
$ and  $\begin{array}{c}
                         G\rightarrow G \\
                         ~x\mapsto x^{-1}
                       \end{array}
$,  are continuous.
An example is, the
additive group of real numbers $(\mathbb{R}, +)$ with the standard topology of $\mathbb{R}$. 
A subgroup of a topological group is also a topological group in the
subspace topology and, for any fixed elements $a$ and $b$ of $G,$
the map $x\mapsto axb$ is a homeomorphism of $G$ onto itself.
Consequently for any open subset $U$ of $G,$ the subsets
$U^{-1}:=\{x^{-1} ~|~ x\in U\},~ xU:=\{xu ~|~u\in U\},$ and
$\displaystyle WU:=\cup_{w\in W}wU$ with the similar definitions for
$Ux$ and $UW,$ are all open in $G,$ where $W\subset G$. 
Similarly, for any closed subset $V$ of $G,$ 
the subsets $xV,~Vx,~V^{-1}$ are closed in $G$
but each of the sets $WV$  and $VW,$ being arbitrary union of
closed subsets of $G,$ need not be closed in $G$.
It is easy to see that any open subgroup $H$ of $G$ is closed in
$G$ because its compliment $G\sim H$ is open as being union of
left co-sets of $H$ in $G$ each of which is an open subset of $G$.
However, a closed subgroup of $G,$ of finite index, is open in $G$
Also if $H$ is a subgroup of $G$ then so is its closure $\bar{H}$.
To see this, note that, for any $x,~y\in \bar{H}$ and the basis elements $B_x\ni x$ and
$B_y\ni y;$ the open sets $xB_y^{-1},$ $B_x y^{-1}$ contain $xy^{-1}$. Since
$B_xB^{-1}_y\cap H \neq \emptyset$ as $B_x$ and $B_y$ intersect
$H$ nontrivially, it follows that $xy^{-1} \in \bar{H}$. 
It is also easy to observe that any  subgroup $H$ of $G$ that contains an open subset $U\ni
1_G$ is always open, as then, it would be union of the open sets
obtained by the co-sets of $U$ in $H$.
An interesting feature of a topological group $G$ is that, it is
homogeneous, i.e., for any pair of points $x$ and $y$ in $G,$
there is a homeomorphism of $G$ sending $x\mapsto y$ (e.g.
$t\mapsto tx^{-1}y$) (see [1, pp.
95-119], [2, p.219], [3, p.16]). The following Theorem \ref{th1} is an easy consequence of these basic notions about the algebraic structure of the real line with respect to the operation of usual addition (see for a `semigroup version proof' of this result in [4] namely the lemma 2.2-2.3).

\begin{theorem}\label{th1}\baselinestretch
Any proper subgroup of the additive group
$\mathbb{R}$ is either its closed subset or its dense subset.
\end{theorem}
\begin{proof}
Let $H$ be a subgroup of $\mathbb{R}$. Assume w.l.o.g that $H$ is not a
closed subset of $\mathbb{R}$. If possible, let us suppose that there is a basis element $(a,b),$
which does not intersect $H$. Then there is a limit point
$x\notin H$ of $H$ in $\mathbb{R}$. So, every basis element
$(c,d)\ni x$ intersects $H$. For each integer $n$ and $t\in H,$
$nx$ and $x+t$ are also limit points of $H,$ since, the maps defined
by $x\mapsto nx$ and $x\mapsto x+t$ are homeomorphisms of
$\mathbb{R}$.
Assume w.l.o.g. that $(a,b)\subsetneq (mx,(m+1)x)$ for some
integer $m$. Then for every positive integer $q,$ satisfying $\displaystyle
q(b-a)>{x},$ there is an integer $p<q$ and a real number $z$
such that $qz=px\in (qa,qb)$ or $z\in(a,b)$. The point $z$ is a
limit point of $H$ since so is $qz$ and the map $z\mapsto qz,$ is a
homeomorphism. But, then $H\cap (a,b)\neq \emptyset$ which is a
contradiction. This completes the proof.
\end{proof}
\begin{remark} Here is more simplified version of the proof of Theorem \ref{th1} based on author's personal communication with Dr. Keerti Vardhan Madahar. Let $x\in  \mathbb{R}-H$ be a limit point of $H$. Choose two points (say $x_1$ and $x_2$) of $H$ which belong to the open set $U = (x-\epsilon/2, x+\epsilon/2)$ of $x$ for some $\epsilon>0$. Now $y = x_1 - x_2$  and all its integral multiples lie in $H$.  Also notice that if $(a, b)$ is an arbitrary basis element of length at least $\epsilon$ then some multiple of $y$ must lie in $(a, b)$ (otherwise there is a positive integer $N$ such that $Ny \leq a<b\leq (N+1)y$ which gives $\epsilon\leq (b-a)\leq y<\epsilon$ a contradiction). That shows $(a, b)$ has a non-trivial intersection with $H$, so $H$ is dense in $\mathbb{R}$.
\end{remark}
The hypothesis of Theorem \ref{th1} does not hold
incase we consider the topological group $\mathbb{R}^n,~n\geq 2$
equipped with the product topology as the subgroup
$\mathbb{Q}^{n-1}\times\{0\}$ is
neither a closed subset of $\mathbb{R}^n$ nor it is a dense subgroup
($\because$ its closure is homeomorphic to $\mathbb{R}^{n-1}$). It is also easy to see that the above result is not true for the multiplicative group $(\mathbb{R}^{\times}, \cdot, <),$ because, the subgroup $\{e^{r}~|~ ~r\in\mathbb{Q}\}$ is neither dense nor a closed subset of $\mathbb{R}^{\times}$.
\begin{remark}
Let $H\neq \{0\}$ be a proper subgroup of the additive group $\mathbb{R}$.  Define $A:=\{|x|: x\in H-\{0\}\}$ and $\alpha:=\inf\{A\}$.

If $H$ is closed in $\mathbb{R}$ then so is $H-\{0\}$ and $\alpha\in \bar{A}\subset \overline{H-\{0\}}=H-\{0\};$ it follows that $\alpha\neq 0$ and $\alpha\in H$. For any $x\in H$ by division algorithm, $|x|=n\alpha+r,~0\leq r<\alpha$ for some positive integer $n$. This means that $r=(|x|-n\alpha)\in H$ such that $r<\alpha$  which is possible only when $r=0$. Hence $H=\alpha\mathbb{Z}$. We have proved that $H$ is cyclic.

Conversely, if $\alpha\neq 0$ then for any $x\in H,$  $|x|=n \alpha+r,~0\leq r<\alpha$ for some positive integer $n$.  We claim that $r=0$. If possible let $r>0$. Using definition of $\inf,$ choose for every $\epsilon>0$ a $y\in H$ s.t. $-n\alpha\in(y,y+\epsilon)$. Take $\epsilon = \alpha+r,$ s.t. $|x|+y\in H$ and $-\alpha=r-\epsilon<|x|+y< r<\alpha$ or $||x|+y|<\alpha$. This is a contradiction since $\alpha=\inf\{A\}$. Thus $r=0$ and $H=\alpha\mathbb{Z}$ which is closed and cyclic subgroup of $\mathbb{R}$.

It is clear that if $\bar{H}=\mathbb{R}$ then $\alpha:=\inf\{A\}=0$. Conversely, if $\alpha=0$ then $H$ is not closed subset of $\mathbb{R}$. Therefore by Theorem \ref{th1}, $H$ is dense in $\mathbb{R}$.

We have proved the following necessary and sufficient conditions for the subgroups of the additive group of reals.
\end{remark}
\begin{theorem}\baselinestretch
Let $H\neq \{0\}$ be a proper subgroup of the additive group $\mathbb{R}$ and $\alpha:=\inf\{|x|: x\in H-\{0\}\}$. Then\\
$(a)$ $H$ is closed in $\mathbb{R}$ if and only if $\alpha\neq 0$. Moreover such a subgroup is cyclic.\\
$(b)$ $H$ is dense in $\mathbb{R}$ if and only if $\alpha=0$.
\end{theorem}
As an application of the preceding result, we have following special case of Kronecker's (1884) approximation theorem.
\begin{theorem}{\sf (Kronecker)}
For an irrational $\alpha,$ the subgroup $\alpha\mathbb{Z}+\mathbb{Z}$  is dense in $\mathbb{R}$.
\end{theorem}
\begin{proof}
  First note that $\alpha\mathbb{Z}+\mathbb{Z}\leq \mathbb{R}$ and so is its closure. If possible let $\alpha\mathbb{Z}+\mathbb{Z}$ be closed in $\mathbb{R}$. Then by preceding theorem, $\alpha\mathbb{Z}+\mathbb{Z}=\beta\mathbb{Z}$ for some nonzero real $\beta$. But then $\mathbb{Z}\subseteq \beta \mathbb{Z}$ and $\alpha\mathbb{Z}\subseteq\beta\mathbb{Z}$. This gives $1=\beta n$ and $\alpha=\beta m$ for some $0\neq m,n\in\mathbb{Z}$ from which we obtain $\beta=\frac{1}{n}$ and $\alpha=\beta m=\frac{m}{n}\in\mathbb{Q}$.  This contradicts the fact that $\alpha\in\mathbb{R}-\mathbb{Q}$.
\end{proof}
\begin{example}
 If we let $S:=\{\sin n~|~n\in\mathbb{Z}\}$ then closure of $S$ in $\mathbb{R}$ is the interval $[-1,1]$. We prove this well known fact (see for detail [5]) using the preceding result. For any real $r\in[-1,1],$ and $\alpha=2\pi$ there exist $s\in [-\pi/2,\pi/2]$ s.t. $\sin s= r$ and integer-sequences  $\{p_n\}$ and $\{q_n\}$ such that $2\pi p_n+q_n\rightarrow s$.
 By continuity of $\sin$ function, we have $\sin (2\pi p_n+q_n)\rightarrow \sin s=r$ or $\sin q_n\rightarrow r$. Thus every point of the interval $[-1,1]$ is limit of some subsequence of the sequence $\{\sin n\}$.
\end{example}




 We now establish a general result concerning the topological groups having the nice property of Theorem \ref{th1}.
\begin{theorem}\label{th3}\baselinestretch
Let $G$ be a topological group such that\\
$(a)$ $G$ has a proper dense subgroup $H$\\
$(b)$ any proper subgroup $K$ of $G$ is such that either $\bar{K}=K$ or $\bar{K}=G$.\\
Then $G$ is either connected or totally disconnected.
\end{theorem}
\begin{proof}
 Let $G$ is not connected and  $C\neq G$ be the connected component of $G$ containing the identity element. If there is a connected open subset $U\neq G$ of $G$ containing a point $x$  of $G$ then so is the open set $x^{-1}U\ni 1_G$ such that $x^{-1}U\cap C\neq \emptyset;$ consequently, $x^{-1}U\cup C$ is connected and hence $C\supset x^{-1}U$. It follows that $C$ is open in $G;$ as $C\leq G$ and $\bar{H}=G$ we see using the hypothesis (b) that $\overline{H\cap C}=\bar{C}=C$ is closed subset of $G$ which gives $\overline{H\cap C}={H\cap C}=C$ or $C\subset H$. But then $H$ is open subgroup of $G$ and since every open subgroup of a topological group is also a closed subset, it follows that $\bar{H}=H$ contradicting the fact that $H$ was assumed to be dense subgroup of $G$. We have proved that there does not exist any connected open subset of $G$ and by (a) none of the singleton set in $G$ can be open. It follows that every proper open subset of $G$ is totally disconnected.
Finally, since $C$ is closed subset of $G$ and if $C\neq G,$ it follows that  $G-C$ is totally disconnected. But then for every $y\in G-C,$ $yC=\{y\}$ which is possible only if $C=\{1_G\}$. 
  This proves that $G$ is totally disconnected.
\end{proof}

In view of the hypothesis of Theorem \ref{th3} through Theorem \ref{th1}, we see that $\mathbb{R}$ is connected in {`usual topology'}, `indiscrete topology', and {`finite-closed topology'} while $\mathbb{R}$ is totally disconnected in `the lower limit topology'. In the latter topolgy we see that  every interval in $\mathbb{R}$ is totally disconnected!

\subsection*{\sf Acknowledgement}
The author is indebted to Dr. Keerti Vardhan Madahar for useful discussion on proof of theorem \ref{th1}.

\subsection*{\sf References:}

[1] L. S. Pontryagin. \emph{Selected Works: Topological Groups Vol
2.} (English translation by A. Brown), Gordon and Breach Science
Publishers, New York, 1986.

[2] N. B\"ourbaki. \emph{Elements of Mathematics: Part I-General
Topology}, Addison Wesley, 1966.

[3] P. J. Higgins. \emph{Introduction to Topological Groups,
Lecture notes 15}, London Mathematical Society. Cambridge
University Press, 1974.

[4] H. Abels and A. Manoussos. Topological generators of abelian Lie groups and hypercyclic finitely generated abelian semigroups of matrices. (2010) pp 1-14.

[5] J. H. Staib and M. S. Demos. On the limit points of the sequence $\sin n$. \emph{Mathematics Magazine}, \textbf{40}:4 (1967) pp 210-213.

\end{document}